\documentclass[a4paper]
{amsart}

\usepackage{amsmath,amssymb,amsthm} 
\usepackage{bbm}
\usepackage{stmaryrd}

\newcommand\A{\mathcal{A}}
\newcommand\F{\mathcal F}
\newcommand\I{\mathbbm{1}}
\renewcommand{\P}{\mathbb{P}}
\newcommand{\M}{\mathcal M}

\newcommand{\D}{\mathcal D}

\renewcommand{\rho}{\varrho}

\newcommand\eps{\varepsilon}

\newcommand\T{0\le t\le T}
\newcommand\1{0\le t\le 1}

\newcommand\p{\mathbb{P}}
\newcommand\N{\mathbb{N}}
\newcommand\E{\mathbb{E}}

\newcommand\RR{\mathbb{R}}

\newcommand\Lone{{L^1 (\Omega)}}
\newcommand\Ltwo{{L^2 (\Omega)}}

\newtheorem{theorem}{Theorem}[section]

\newtheorem{lemma}[theorem]{Lemma}

\theoremstyle{definition}

\DeclareMathOperator{\conv}{conv}

\title[A short Proof of the Doob-Meyer Theorem]{A short Proof of the Doob-Meyer Theorem}

\begin{document}
\author {Mathias Beiglb\"ock, Walter Schachermayer, Bezirgen Veliyev}
\thanks{The first author acknowledges support from the Austrian Science
Fund under grant  P21209. The second author acknowledges support from the Austrian Science Fund under grant P19456, from
the Vienna Science and Technology Fund under grant MA13, and from the ERC Advanced Grant.
 The third author acknowledges  support from the Austrian Science Fund under grant P19456.}

\maketitle
\begin{abstract}
Every submartingale $S$ of class $D$ has a unique Doob-Meyer decomposition $S=M+A$, where $M$ is a martingale and $A$ is a predictable increasing process starting at $0$.

We provide a short and elementary prove of the Doob-Meyer decomposition theorem. Several previously known  arguments are included to keep the paper self-contained.

\end{abstract}
\section{Introduction}
Throughout this article we fix a probability space $(\Omega, \F, \p)$ and  a right-continuous complete filtration $(\F_t)_{\T}$.


 An adapted process $(S_t)_{\T}$ is of class $D$ if the family of random variables $ S_{\tau}  $ where $\tau $ ranges through all stopping times is uniformly integrable (\cite{Meye62}).

 The purpose of this paper is to give a short and elementary proof of the following
\begin{theorem}[Doob-Meyer]\label{DoobMeyer}
  Let $S=(S_t)_{\T}$ be a c\`{a}dl\`{a}g submartingale of class $D$. Then, $S$ can be written in a unique way in the form 
\begin{align}\label{TheDecomposition}S=M+A\end{align} 
where $M$ is a martingale and $A$ is a predictable increasing process starting at $0$. 
\end{theorem}

Doob \cite{Doob53} noticed that in discrete time an integrable process $S=(S_n)_{n=1}^\infty$ can be uniquely represented as the sum of a martingale $M$ and a predictable process $A$ starting at $0$; in addition, the process $A$ is increasing iff $S$ is a submartingale. The continuous time analogue, Theorem \ref{DoobMeyer}, goes back to Meyer \cite{Meye62,Meye63}, who introduced the class $D$ and proved that  every submartingale $S=(S_t)_{\T}$ can be decomposed in the form \eqref{TheDecomposition}, where $M$ is a martingale and $A$ is a \emph{natural} process. The modern formulation is due to Dol{\'e}ans-Dade \cite{Dole67, Dole68} who obtained that an increasing process is  natural iff it is  predictable. Further proofs of Theorem \ref{DoobMeyer} were given by Rao \cite{Rao69}, Bass \cite{Bass96} and Jakubowski \cite{Jaku05}. 

 Rao works with the
$\sigma(L^1, L^{\infty})$-topology and applies the Dunford-Pettis compactness criterion to obtain the desired  continuous time decomposition as a weak-$L^1$ limit from discrete approximations. To obtain that $A$ is predictable one then invokes the theorem of Dol{\'e}ans-Dade.

Bass gives a more elementary proof based on the dichotomy between predictable and totally inaccessible stopping times. 

Jakubowski proceeds as Rao, but notices that predictablity of the process $A$ can also be obtained through an application of Komlos' Lemma \cite{Koml67}.

The proof presented subsequently combines ideas from \cite{Jaku05} and \cite{BeSV10} to construct the continuous time decomposition 
using a suitable Komlos-type lemma. 
\section{Proof of Theorem \ref{DoobMeyer}}
The proof of uniqueness is standard and we have nothing to add here; see for instance \cite[Lemma 25.11]{Kall02}.

For the remainder of this article we work under the assumptions of Theorem \ref{DoobMeyer} and fix $T=1$ for simplicity.


Denote by $\D_n$ and $\D$ the set of  $n$-th resp.\ all dyadic numbers $j/2^n$ in the interval $[0,1].$ 
For each $n$, we consider the discrete time Doob decomposition of the sampled process $S^n=(S_t)_{t\in\D_n}$, that is, we define $A^n, M^n$ by $A^n_0:=0$,
\begin{align} A^n_{t}-A^n_{t-1/2^n}&:=\E[S_{t}-S_{t-1/2^n}|\F_{t-1/2^n}] \ \mbox{and} 
\\ M^n_t& := S_t-A^n_t
\end{align}
so that $(M^n_{t})_{t\in\D_n}$ is a martingale and $(A^n_{t})_{t\in\D_n}$ is  predictable with respect to $(\F_{t})_{t\in\D_n}$. 

The idea of the proof is, of course, to obtain the continuous time decomposition \eqref{TheDecomposition} as a limit, or rather, as an accumulation point of the processes $M^n,A^n,n\geq 1$. 

Clearly, in infinite dimensional spaces a (bounded) sequence need not have a convergent subsequence. 
As a substitute for the Bolzano-Weierstrass Theorem we establish the Komlos-type Lemma \ref{EasyKomlos} in Section \ref{KLS}.

In order to apply this auxiliary result, we require that the sequence $(M_1^n)_{n\geq 1}$ is uniformly integrable. This follows from the class $D$ assumption as shown by \cite{Rao69}. To keep the paper self-contained, we provide a proof in Section \ref{UIS}.


Finally, in Section \ref{LIS}, we obtain the desired decomposition by passing to a limit of the discrete time versions.  As the Komlos-approach guarantees convergence in a strong sense, predictability of the process $A$ follows rather directly from the predictability of the approximating processes.
This idea is taken from \cite{Jaku05}.
\subsection{Komlos' Lemma}\label{KLS}


Following Komlos \cite{Koml67}\footnote{Indeed,  \cite{Koml67} considers Cesaro sums along subsequences rather then arbitrary convex combinations. But for our purposes, the more modest conclusion of Lemma \ref{EasyKomlos} is sufficient.}, 
it is sometimes possible to obtain an accumulation point of a bounded sequence in an infinite dimensional space if appropriate convex combinations are taken into account. 

\medskip

A particularly simple result of this kind holds true if $(f_n)_{n\geq1} $ is a bounded sequence in a Hilbert space. In this case $$\textstyle{A=\sup_{n\geq1} \inf\{\|g\|_2:g\in\conv\{f_n, f_{n+1},\ldots\}\}}$$ is finite and 
 for each $n$ we  may pick some $g_n\in \conv\{f_n, f_{n+1},\ldots\}$ such that $ \|g_n\|_2\leq A+1/n$. If $n$ is sufficiently large with respect to $\eps>0$, then $\|(g_k+g_m)/2\|_2>A-\eps$ for all $m,k\geq n$ and hence
$$ \|g_k-g_m\|_2^2=2 \|g_k\|_2^2+2\|g_m\|_2^2- \|g_k+g_m\|_2^2
\leq 4(A+\tfrac1n)^2-4(A-\eps)^2.$$ By completeness, $(g_n)_{n\geq1}$  converges in $\|.\|_2$.



\medskip

By a straight forward truncation procedure this Hilbertian Komlos-Lemma yields an $L^1$-version which we will need subsequently.\footnote{Lemma \ref{EasyKomlos} is also a trivial consequence of Komlos' original result \cite{Koml67} or other related results that have been established through the years. Cf.\ \cite[Chapter 5.2]{KaSa09} for an overview.}
\begin{lemma}\label{EasyKomlos}
 Let $(f_{n})_{n\geq1}$ be a uniformly integrable  sequence of functions on a probability space $(\Omega, \F,\P).$ 
 Then there exist functions  $g_n\in\conv(f_{n},f_{n+1},\dots)$ such that $(g_n)_{n\geq 1}$ 
 converges in $\|.\|_\Lone.$
\end{lemma}
\begin{proof}
For $i,n\in\N$ set $f_{n}^{(i)}:=f_n \I_{\{|f_n|\leq i\}}$ such that $f_{n}^{(i)}\in \Ltwo$. 

We claim that there exist
for every $n$ convex weights $\lambda_n^{n}, \ldots, \lambda_{N_n}^{n}$ such that the functions
$ \lambda_n^{n} f_n^{(i)} + \ldots+\lambda_{N_n}^{n} f_{N_n}^{(i)}$
converge in $\Ltwo$ for every $i\in\N$. 

To see this, one first uses the Hilbertian lemma to find convex weights $\lambda_n^{n}, \ldots, \lambda_{N_n}^{n}$ such that   
$( \lambda_n^{n} f_n^{(1)} + \ldots+\lambda_{N_n}^{n} f_{N_n}^{(1)})_{n\geq 1}$ converges. In the second step, one applies the lemma to the sequence 
$( \lambda_n^{n} f_n^{(2)} + \ldots+\lambda_{N_n}^{n} f_{N_n}^{(2)})_{n\geq 1} $, to obtain convex weights which work for the first two sequences. Repeating this procedure inductively we obtain sequences of convex weights which work for the first $m$ sequences.  
Then a standard diagonalization argument yields the claim.

By uniform integrability, $\lim_{i\to \infty}\| f^{(i)}_n- f_n\|_1=0$, uniformly with respect to $n$.  
Hence, once again, uniformly with respect to $n$, 
$$ \textstyle\lim_{i\to\infty}\|  (\lambda_n^{n} f_n^{(i)} + \ldots+\lambda_{N_n}^{n} f_{N_n}^{(i)})-(\lambda_n^{n} f_n + \ldots+\lambda_{N_n}^{n} f_{N_n})\|_1= 0.$$ 
Thus $(\lambda_n^{n} f_n + \ldots+\lambda_{N_n}^{n} f_{N_n})_{n\geq 1}$  is a Cauchy sequence in $\Lone$.
 \end{proof}


\subsection{Uniform integrability of the discrete approximations.}\label{UIS}
\begin{lemma}\label{UIlemma}
The sequence $(M^n_1)_{n\geq1}$ is uniformly integrable.
\end{lemma}
\begin{proof}
Subtracting $\E[S_{1}|\F_t]$ from $S_t$ we may assume that $S_1=0$ and $S_t \leq 0$ for all $0 \leq t \leq 1$.
Then $M_{1}^n=-A_{1}^n,$ and for every $(\F_t)_{t\in\D_n}$-stopping time $\tau$
\begin{align}\label{eq0}
S_{\tau}^n=-\E[A_{1}^n|\F_{\tau}]+ A_{\tau}^n.
\end{align}
We claim that $(A_{1}^n)_{n=1}^{\infty}$ is uniformly integrable. 
For $c>0$, $n \geq 1$ define  
$$\tau_n(c)=\inf \big \{ (j-1)/{2^n}: A^n_{{j}/{2^n}} > c \big\} \wedge 1 .$$
 From $A_{\tau_{n}(c)}^n \leq c$ and \eqref{eq0} we obtain
 $ S_{\tau_{n}(c)} \leq -E[A_{1}^n|\F_{\tau_{n}(c)}]+c.$
Thus, 
$$\int_{\{A_{1}^n >c \}} A_{1}^n \,d\P=\displaystyle\int_{\{ \tau_n(c)<1\}}\E[A_{1}^n|\F_{\tau_{n}(c)}] \,d\P
 \leq c \, \p\big[ \tau_{n}(c) <1 \big]- \int_{ \{ \tau_n(c) <1 \} } S_{\tau_{n}(c)} \, d\P.$$
Note $\{\tau_n(c) <1\} \subseteq \{\tau_n(\frac{c}{2}) <1\},$ hence, by \eqref{eq0}
\begin{align*}
 \int_{{\{\tau_n(\frac{c}{2}) <1\} }}\!\!\!\!\!\! - S_{\tau_{n}(\frac{c}{2})}\,d\P &=
\int_{\{\tau_n(\frac{c}{2}) <1\} }\!\!A_{1}^n-A^n_{\tau_{n}(\frac{c}{2})} \, d\P  \\
& \geq
\int_{\{\tau_n(c) <1\} }\!\!A_{1}^n-A^n_{\tau_{n}(\frac{c}{2})} \, d\P   \geq
\frac{c}{2} \, \p[\tau_n(c) <1]. 
\end{align*}
Combining the above inequalities we obtain
\begin{align}\label{eq1}
\int_{ \{A_{1}^n >c \} } A_{1}^n \,d\P \leq -2 \int_{\{\tau_n(\frac{c}{2}) <1\} } S_{\tau_{n}(\frac{c}{2})} \, d\P- 
\int_{ \{ \tau_n(c) <1 \} } S_{\tau_{n}(c)} \, d\P. 
\end{align}
On the other hand
\begin{align*}
\p[\tau_n(c)<1]&=  \p[A_{1}^n>c] 
\leq {\E[A_{1}^n]}/{c} = {-\E[M_{1}^n]}/{c}
  = {-\E[S_{0}]}/{c}, 
\end{align*}
hence, as $c\to\infty$, $\p[\tau_n(c)<1]$ goes to $0$, uniformly in $n$.
As $S$ is of class $D$, \eqref{eq1} implies that the sequence $(A_{1}^n)_{ n\geq 1}$ is uniformly integrable and 
hence $(M_{1}^n)_{ n\geq 1}=(S_{1}-A^n_1)_{ n\geq 1}$ is uniformly integrable as well.\end{proof}
\subsection{The limiting procedure.}\label{LIS}

For each $n$, extend $M^n$ to a (c\`adl\`ag) martingale on $[0,1]$ by setting $M_t^n:=\E[M_1^n|\F_t]$.
By Lemma \ref{EasyKomlos} and Lemma \ref{UIlemma} there exist $M \in \Lone$  and for each $n$ convex weights $\lambda_n^{n}, \ldots, \lambda_{N_n}^{n}$ such that with  
\begin{align} \label{weights2}
 \M^n:=\lambda_{n}^{n} M^{n} + \ldots+\lambda_{N_n}^{n}  M^{N_n}
\end{align}
we have $\M^n_1\to M$ in $\Lone.$ Then, by Jensen's inequality, 
$\M^n_t\to M_t:=\E[M|\F_t]$  for all $t\in[0,1].$
 For each $n\geq 1$ we extend $A^n$ to $[0,1]$ by
 \begin{align}
 A^n&:=\textstyle{\sum_{t\in \D_n} A^{n}_t\I_{(t-1/2^n, t ] } }\\
 \mbox{and set }\quad \A^n&:=\lambda_{n}^{n} A^{n} + \ldots+\lambda_{N_n}^{n}  A^{N_n},
 \end{align} 
 where we use the same convex weights as in \eqref{weights2}. 
Then the c\`adl\`ag process $$(A_t)_{\1}:=(S_t)_{\1}-(M_t)_{\1}$$
satisfies for every $t\in\D$ 
$$\A_t^n=(S_t-\M_t^n)\ \to\ (S_t-M_t)=A_t\quad \mbox{in $\Lone$.}$$ 
 Passing to a subsequence which we denote again by $n$, we obtain that convergence holds also almost surely. Consequently, $A$ is almost surely increasing on $\D$ and, by right continuity,  also on $[0,1]$.

As the  processes ${A}^n$ and $\A^n$ are left-continuous and adapted, they are predictable. 
To obtain that $A$ is predictable, we show that for a.e.\ $\omega$ and every $t \in [0,1]$
\begin{equation} \label{pred1}
 \textstyle \limsup_{n} \A_{t}^n(\omega)=A_{t}(\omega).
\end{equation}
If  $f_n,f:[0,1]\to \RR$ are increasing functions such that $f$ is right continuous and $ \lim_{n } f_n(t) = f(t)$ for $t\in\D,$ then
\begin{align}\label{Aleq}
\textstyle{\limsup_{n }} f_n (t)\leq f(t) \ \mbox {for all $t\in[0,1]$ and }\\
\textstyle{\lim_{n }} f_n(t) = f(t) \mbox { if $f$ is continuous at $t$.} \label{Asame}
\end{align}
Consequently, \eqref{pred1} can only be violated at discontinuity points of $A.$ As $A$ is c\`adl\`ag, every path of $A$ can
have only finitely many jumps larger  than  $1/k$ for $k\in\N$. It follows that the points of discontinuity of $A$ can be
exhausted by a countable sequence of stopping times, and therefore it is sufficient to prove 
$\limsup_{n } \A^n_{\tau}=A_{\tau}$ for every stopping time $\tau.$

By \eqref{Aleq}, $\limsup_{n } \A^n_{\tau}\leq A_{\tau}$ and as $\A_{\tau}^n \leq \A_{1}^n \to A_{1} \mbox{ in }  \Lone$ 
we deduce from Fatou's Lemma  that
$$\textstyle{\liminf_{n }\E \big[{A}_{\tau}^n \big] \leq \limsup_{n } \E \big[\A_{\tau}^n \big] \leq \E\big[ \limsup_{ n} \A_{\tau}^n \big] \leq \E\big[A_{\tau}\big].}$$ 
Therefore it suffices to prove $\lim_{n  } \E[{A}_{\tau}^n]=\E[A_{\tau}].$
For $n \geq 1$ set $$\sigma_{n}:=\inf\{t\in \D_n: t\geq \tau\}.$$ 
Then $ {A}_{\tau}^n=A^{n}_{\sigma_{n}}$  and $\sigma_n \downarrow \tau$. 
Using that $S$ is of class $D$, we obtain
$$ \E[{A}_{\tau}^n]=\E[A_{\sigma_{n}}^n]=\E [S_{\sigma_{n}}] -\E[M_{0}]
\to \E[S_{\tau}]-\E[M_{0}]=\E[A_{\tau}].$$
 
\bibliographystyle{alpha}
\bibliography{Bezirgen}

\end{document}